\documentclass[11pt, reqno]{amsart}
\usepackage{amsmath, amsthm, amscd, amsfonts, amssymb, graphicx, color}
\usepackage[bookmarksnumbered, colorlinks, plainpages]{hyperref}
\input{mathrsfs.sty}

\newtheorem{theorem}{Theorem}[section]

\newtheorem{proposition}[theorem]{Proposition}
\newtheorem{corollary}[theorem]{Corollary}
\theoremstyle{definition}
\newtheorem{definition}[theorem]{Definition}

\theoremstyle{remark}
\newtheorem{remark}[theorem]{Remark}

\begin{document}

\title[Epigraph of Operator Functions]{Epigraph of Operator Functions}

\author[ M. Kian]{ Mohsen  Kian }

\address{Mohsen Kian:\ \
 Department of Mathematics, Faculty of Basic Sciences, University of Bojnord, P. O. Box
1339, Bojnord 94531, Iran}
\address{School of Mathematics, Institute for Research in Fundamental Sciences (IPM), P.O. Box: 19395-5746, Tehran, Iran.}
\email{\textcolor[rgb]{0.00,0.00,0.84}{kian@ub.ac.ir \ and \ kian@member.ams.org}}




\begin{abstract}
It is known that a real function $f$ is convex if and only if the set
$$\mathrm{E}(f)=\{(x,y)\in\mathbb{R}\times\mathbb{R};\ f(x)\leq y\},$$
the  epigraph of  $f$  is a convex set in $\mathbb{R}^2$.
We state an extension of this result  for operator convex functions and $C^*$-convex sets as well as operator log-convex functions and $C^*$-log-convex sets. Moreover, the $C^*$-convex hull of a Hermitian matrix has been  represented in terms of its eigenvalues.
\end{abstract} \maketitle
\section{Introduction and preliminaries}
\noindent
Throughout the paper,  assume that $\mathcal{B}(\mathcal{H})$ is the $C^*$-algebra of all bounded linear operators on a Hilbert space $\mathcal{H}$. If $dim(\mathcal{H})=n$, then we identify  $\mathcal{B}(\mathcal{H})$ with $\mathbb{M}_n$, the algebra of all $n\times n$ matrices. We denote by $\mathbb{H}_n$ the set of all Hermitian matrices in $\mathbb{M}_n$. An operator $X\in\mathcal{B}(\mathcal{H})$ is called positive (denoted by $X\geq0$) if $\langle Xa,a\rangle\geq 0$ for every $a\in\mathcal{H}$. If in addition $X$ is invertible, then it is called strictly positive(denoted by $X>0$). We denote by  $\mathcal{B}(\mathcal{H})^+$  the set of all strictly positive operators on  $\mathcal{H}$.

For a real interval $J$, we mean by $sp(J)$ the set of all self-adjoint operators on $\mathcal{H}$ whose spectra
 are contained in $J$. A continuous function $f:J\to\mathbb{R}$ is said to be operator convex if $f\left(\frac{X+ Y}{2}\right)\leq \frac{f(X)+ f(Y)}{2}$ for all $X,Y\in sp(J)$.
It is well known that \cite{FMPS,HP}  $f$ is operator convex if and only if  the Jensen operator inequality
\begin{align}\label{jensen}
f(C^*XC)\leq C^*f(X)C
\end{align}
holds  for every isometry $C$ and every $X\in sp(J)$. A continuous function $f:(0,\infty)\to(0,\infty)$ is called operator log-convex \cite{AH,kian-dr} if $f\left(\frac{X+ Y}{2}\right)\leq f(X)\sharp f(Y)$ for  all strictly positive operators $X,Y$, where the  geometric  mean $\sharp$ is  defined by $X\sharp Y=X^{\frac{1}{2}}\left(X^{-\frac{1}{2}}YX^{-\frac{1}{2}}\right)^{\frac{1}{2}}X^{\frac{1}{2}}$ for all strictly positive operators $X$ and $Y$, see for example \cite{FMPS}. In the case where $f$ is operator log-convex a sharper inequality than \eqref{jensen} is valid as
\begin{align}\label{jensen2}
f\left(\sum_{i=1}^{n}C_i^*X_iC_i\right)\leq \left(\sum_{i=1}^{n}C_i^*f(X_i)^{-1}C_i\right)^{-1}
\end{align}
 for all $X_1,\cdots,X_n>0$ and all  $C_1,\cdots,C_n\in\mathcal{B}(\mathcal{H})$  with $\sum_{i=1}^{n}C_i^*C_i=I$, see \cite[Corollary 3.13]{kian-dr}.

 A set $\mathcal{K}\subseteq\mathcal{B}(\mathcal{H})$ is called $C^*$-convex if $X_1,\cdots,X_n\in \mathcal{K}$ and
  $C_1,\cdots,C_n\in\mathcal{B}(\mathcal{H})$ with $\sum_{i=1}^{n}C_i^*C_i=I$ implies that $\sum_{i=1}^{n}C_i^*X_iC_i\in\mathcal{K}$. This kind of convexity has been introduced by  Loebl and Paulsen \cite{Loebl} as a non-commutative generalization of linear convexity and has been studied by many authors, see e.g.  \cite{FZ,Morenz, WW} and references therein. Typical examples of $C^*$-convex  sets are $\{T\in\mathcal{B}(\mathcal{H}):\,0\leq T\leq I\}$
 and $\{T\in\mathcal{B}(\mathcal{H});\,\|T\|\leq M\}$ for a fix scalar $M>0$.
 It is evident that the $C^*$-convexity of a set $\mathcal{K}$ in $\mathcal{B}(\mathcal{H})$  implies its  convexity in the usual sense. For if $X,Y\in\mathcal{K}$ and $\lambda\in[0,1]$, then with $C_1=\sqrt{\lambda}I$ and $C_2=\sqrt{1-\lambda}I$ we have $C_1^*C_1+C_2^*C_2=I$ and
  $$\lambda X+(1-\lambda)Y= C_1^*XC_1+C_2^*YC_2\in\mathcal{K}.$$
  But the converse is not true in general. For example   if $A\geq0$, then $$[0,A]=\left\{X\in\mathcal{B}(\mathcal{H});\ \ 0\leq X\leq A\right\}$$
   is convex but not $C^*$-convex \cite{Loebl}.
   The concept of $C^*$-convexity can be generalized to the sets which have a $\mathcal{B}(\mathcal{H})$-module structure. Assume that  $\mathcal{M}$ is a $\mathcal{B}(\mathcal{H})$-module. We say that a subset $\mathcal{K}$ of $\mathcal{M}$ is  $C^*$-convex whenever $X_1,\cdots,X_n\in\mathcal{K}$, $C_1,\cdots,C_n\in\mathcal{B}(\mathcal{H})$ and $\sum_{i=1}^{n}C_i^*C_i=I$ implies that $\sum_{i=1}^{n}C_i^*X_iC_i\in\mathcal{K}$. For example, let
   $$\mathcal{M}=\left\{(X_1,\cdots,X_k);\ X_j\in\mathcal{B}(\mathcal{H}),\ j=1,\cdots,k\right\}.$$
    Then $\mathcal{M}$ is a $\mathcal{B}(\mathcal{H})$-module under
   \begin{align*}
     \mathcal{M}\times\mathcal{B}(\mathcal{H})&\to\mathcal{M}\qquad ((X_1,\cdots,X_k),T)\mapsto (X_1T,\cdots,X_kT)\\
     \mathcal{B}(\mathcal{H})\times\mathcal{M}&\to\mathcal{M}\qquad (S,(X_1,\cdots,X_k))\mapsto (SX_1,\cdots,SX_k).
   \end{align*}
Now,  $\mathcal{K}\subseteq\mathcal{M}$ is called $C^*$-convex if  $X_i=(X_{i1},\cdots,X_{ik})\in\mathcal{K}$,  $C_i\in\mathcal{B}(\mathcal{H})$\ $(i=1,\cdots,n)$ and $\sum_{i=1}^{n}C_i^*C_i=I$ implies that
$$\sum_{i=1}^{n}C_i^*X_iC_i
=\left(\sum_{i=1}^{n}C_i^*X_{i1}C_i,\cdots,\sum_{i=1}^{n}C_i^*X_{ik}C_i\right)\in\mathcal{K}.$$
As an example, it is easy to see that $\mathcal{K}=\left\{(X_1,\cdots,X_k)\in\mathcal{M};\ \ 0\leq X_j\leq I,\ j=1,\cdots,k\right\}$ is $C^*$-convex.

The epigraph of a real function $f$ is defined to be the set
$$\mathrm{E}(f)=\{(x,y)\in\mathbb{R}\times\mathbb{R};\ f(x)\leq y\}.$$
It is known that $f$ is a convex function if and only if $\mathrm{E}(f)$ is a convex set in $\mathbb{R}^2$.
The main purpose of this paper is to present  this result for operator functions. In particular, we give the connection between operator convex functions and $C^*$-convex sets as well as operator log-convex functions and $C^*$-log-convex sets. It is also shown that the $C^*$-convex hull of a Hermitian matrix can be represented in terms of its eigenvalues.

\section{Main Result}
For a continuous real function  $f:J\to\mathbb{R}$, we define the operator epigraph of $f$ by
\begin{eqnarray*}
  \mathrm{OE}(f):=\{(X,Y)\in sp(J)\times\mathcal{B}(\mathcal{H});\  f(X)\leq Y\}.
\end{eqnarray*}
The next result gives the connections between operator convex functions and $C^*$-convex sets.
\begin{theorem}\label{t1}
A continuous  function $f:J\to\mathbb{R}$ is operator convex if and only if  $\mathrm{OE}(f)$ is $C^*$-convex.
\end{theorem}
\begin{proof}
  Let $f$ be operator convex. Let $(X_i,Y_i)\in \mathrm{OE}(f)$ \ $(i=1,\cdots,n)$ and $C_i\in\mathcal{B}(\mathcal{H})$ with $\sum_{i=1}^{n}C_i^*C_i=I$. Therefore $f(X_i)\leq Y_i$ \ $(i=1,\cdots,n)$ and so we have by the Jensen operator inequality that
  \begin{eqnarray*}
    f\left(\sum_{i=1}^{n}C_i^*X_i C_i\right)\leq\sum_{i=1}^{n}C_i^*f(X_i)C_i\leq\sum_{i=1}^{n}C_i^*Y_iC_i.
  \end{eqnarray*}
In other words, $\sum_{i=1}^{n}C_i^*(X_i,Y_i)C_i\in \mathrm{OE}(f)$ and so $\mathrm{OE}(f)$ is $C^*$-convex.

Conversely, assume that $\mathrm{OE}(f)$ is $C^*$-convex. For all $X_1,\cdots,X_n\in\mathcal{B}(\mathcal{H})$, we have  from the definition of $\mathrm{OE}(f)$ that $(X_i,f(X_i))\in \mathrm{OE}(f)$. If $C_i\in\mathcal{B}(\mathcal{H})$ with   $\sum_{i=1}^{n}C_i^*C_i=I$,
then $\sum_{i=1}^{n}C_i^*(X_i,f(X_i))C_i\in \mathrm{OE}(f)$ by the $C^*$-convexity of $\mathrm{OE}(f)$. It follows that
\begin{eqnarray*}
 f\left(\sum_{i=1}^{n}C_i^*X_iC_i\right)\leq\sum_{i=1}^{n}C_i^*f(X_i)C_i,
\end{eqnarray*}
which means that $f$ is operator convex.
\end{proof}
Let $\{f_{\alpha};\ \alpha\in\Gamma\}$     be a family of operator convex functions and $M_{\alpha}\in \mathbb{R}$ for every $\alpha\in\Gamma$. By Theorem \ref{t1}, the set  $\{X\in\mathcal{B}(\mathcal{H}); \ f_{\alpha}(X)\leq M_{\alpha},  \ \forall\alpha\in\Gamma\}$ is $C^*$-convex. For example, consider the family $f_{\alpha}$ where $f_{\alpha}(t)=t^{\alpha}$   and $\alpha\in[1,2]$. Then $\{X\in\mathcal{B}(\mathcal{H}); \ X^{\alpha}\leq M_{\alpha},  \ 1\leq \alpha\leq2\}$ is $C^*$-convex.

The Choi--Davis--Jensen inequality for an operator convex function $f:J\to\mathbb{R}$ asserts  that $f(\Phi(X))\leq\Phi(f(X))$ for every unital positive linear mapping $\Phi$ on $\mathcal{B}(\mathcal{H})$ and every $X\in sp(J)$, see \cite{Choi, Davis, FMPS}.  Motivating by this result, we state a characterization for $C^*$-convex sets in $\mathbb{H}_n$ using positive linear mappings.
\begin{theorem}\label{th2}
  If  $\mathcal{K}\subseteq\mathbb{H}_n$,  then the followings are equivalent:
  \begin{enumerate}
    \item $\mathcal{K}$ is $C^*$-convex;\\
    \item $\sum_{i=1}^{m}\Phi_i(X_i)\in \mathcal{K}$  for every $X_i\in \mathcal{K}$, \ $(i=1,\cdots,m)$ and every unital family $\{\Phi_i; \ i=1,\cdots,m\}$  of positive linear mappings on $\mathbb{M}_n$.
      \end{enumerate}
\end{theorem}
\begin{proof}
Assume that $\mathcal{K}\subseteq\mathbb{H}_n$ is   $C^*$-convex. First note that if $X\in\mathcal{K}$ and if $\lambda$ is an eigenvalue of $X$, then $\lambda I\in\mathcal{K}$. Indeed, it follows from the spectral decomposition that there exists a unitary $U$ such that   $U^*XU$ is  a matrix such that $\lambda$ is its $k,k$ entry.  Then  $\lambda I=\sum_{i=1}^{n}E_{ki}^*U^*XUE_{ki}\in\mathcal{K}$, where $\{E_{ij}\}$ is the system of unit matrices.
Now let   $X_i\in \mathcal{K}$, \ $(i=1,\cdots,m)$ and let $\{\Phi_i; \ i=1,\cdots, m\}$  be a unital family  of positive linear mappings on $\mathbb{M}_n$. Assume that $X_i=\sum_{j=1}^{n}\lambda_{ij}P_{ij}$ be the spectral decomposition of $X_i$ for $i=1,\cdots, m$ so that   $\lambda_{ij} I\in \mathcal{K}$ for all $i,j$. Therefore
\begin{align*}
 \sum_{i=1}^{m}\Phi_i(X_i)=\sum_{i=1}^{m}\Phi_i\left(\sum_{j=1}^{n}\lambda_{ij}P_{ij}\right)
 = \sum_{i=1}^{m}\sum_{j=1}^{n}\lambda_{ij}\Phi_i(P_{ij})
 = \sum_{i=1}^{m}\sum_{j=1}^{n}C_{ij}^*\lambda_{ij}C_{ij},
\end{align*}
where $C_{ij}=\sqrt{\Phi_i(P_{ij})}$. Taking into account that
\begin{eqnarray*}
\sum_{i=1}^{m}\sum_{j=1}^{n}C_{ij}^*C_{ij}=\sum_{i=1}^{m}\sum_{j=1}^{n}\Phi_i(P_{ij})=
\sum_{i=1}^{m}\Phi_i\left(\sum_{j=1}^{n}P_{ij}\right)=\sum_{i=1}^{n}\Phi_i(I)=I,
\end{eqnarray*}
we get by the $C^*$-convexity of $\mathcal{K}$ that $\sum_{i=1}^{m}\Phi_i(X_i)\in\mathcal{K}$.

Conversely, let $X_1,\cdots,X_m\in\mathcal{K}$ and $C_i\in\mathbb{M}_n$ with $\sum_{i=1}^{m}C_i^*C_i=I$. Define positive linear mappings $\Phi_i$ on $\mathbb{M}_n$ by $\Phi_i(X)=C_i^*XC_i$, \ $(i=1,\cdots,m)$. Then $\{\Phi_i; \ i=1,\cdots, m\}$ is a unital family and so
$$\sum_{i=1}^{m}C_i^*X_iC_i=\sum_{i=1}^{m}\Phi_i(X_i)\in\mathcal{K},$$
i.e., $\mathcal{K}$ is $C^*$-convex.
  \end{proof}
The famous Choi--Davis--Jensen inequality which is a characterization of operator convex functions, can be derived from Theorem \ref{t1} and Theorem \ref{th2}.
\begin{corollary}
 A continuous function $f:J\to\mathbb{R}$ is operator convex if and only if
  \begin{align}\label{q1}
  f\left(\sum_{i=1}^{m}\Phi_i(X_i)\right)\leq \sum_{i=1}^{m}\Phi_i(f(X_i))
 \end{align}
  for every $X_i\in sp(J)$ and every
 unital family $\{\Phi_i\}_{i=1}^{m}$  of positive linear mappings on $\mathbb{M}_n$.
\end{corollary}
\begin{proof}
  Let $f:J\to\mathbb{R}$ be operator convex. Then $\mathrm{OE}(f)$ is $C^*$-convex by Theorem \ref{t1}. For every $X_i\in sp(J)$, we have $(X_i,f(X_i))\in \mathrm{OE}(f)$. Theorem \ref{th2} then implies that $\sum_{i=1}^{m}\Phi_i(X_i,f(X_i))\in \mathrm{OE}(f)$ for every unital family $\{\Phi_i\}_{i=1}^{m}$   of positive linear mappings. Hence \eqref{q1} holds true. Conversely, assume that \eqref{q1} is valid. Let $(X_i,Y_i)\in \mathrm{OE}(f)$,\ $(i=1,\cdots, m)$ so that $f(X_i)\leq Y_i$. If $\{\Phi_i\}_{i=1}^{m}$ is a  unital family of positive linear mappings, then
  \begin{align*}
    f\left(\sum_{i=1}^{m}\Phi_i(X_i)\right)&\leq \sum_{i=1}^{m}\Phi_i(f(X_i))\quad (\mbox{by \eqref{q1}})\\
    &\leq \sum_{i=1}^{m}\Phi_i(Y_i)\qquad (\mbox{by  $f(X_i)\leq Y_i$}).
  \end{align*}
  This concludes that $\sum_{i=1}^{m}\Phi_i(X_i,Y_i)\in \mathrm{OE}(f)$. Theorem \ref{th2} now implies that $\mathrm{OE}(f)$ is $C^*$-convex and so $f$ is operator convex by Theorem \ref{t1}.

\end{proof}

\begin{remark}
Let $\mathcal{K}\subseteq\mathcal{B}(\mathcal{H})^+$ be  $C^*$-convex  and $0\in \mathcal{K}$. If $C\in\mathcal{B}(\mathcal{H})$ is a contraction, then $C^*XC\in\mathcal{K}$ for every $X\in\mathcal{K}$. To see this, put $D=(I-C^*C)^{\frac{1}{2}}$. Then $C^*C+D^*D=I$ and so $C^*XC=C^*XC+D^*0D\in\mathcal{K}$ for every $X\in\mathcal{K}$. It follows that if $\sum_{i=1}^{n}X_i\in\mathcal{K}$, then $X_i\in\mathcal{K}$ for $i=1,\cdots,n$.  For if $X,Y\in\mathcal{K}$, put $C_1=(X+Y)^{-\frac{1}{2}}X^{\frac{1}{2}}$ and $C_2=(X+Y)^{-\frac{1}{2}}Y^{\frac{1}{2}}$ which are contractions and
\begin{eqnarray*}
  X=C_1^*(X+Y)C_1 \ \ \ \ \ Y=C_2^*(X+Y)C_2.
\end{eqnarray*}
\end{remark}
\begin{definition}
  We say that  a set $\mathcal{K}\subseteq\mathcal{B}(\mathcal{H})$ is $C^*$-log-convex if $\left(\sum_{i=1}^{n}C_i^*X_i^{-1}C_i\right)^{-1}\in\mathcal{K}$ for all $X_i\in\mathcal{K}$ and  $C_i\in\mathcal{B}(\mathcal{H})$ with $\sum_{i=1}^{n}C_i^*C_i=I$.
\end{definition}
 If $M$ is a positive scalar, then $\{X\in\mathcal{B}(\mathcal{H}); \ 0<X\leq M\}$ is an obvious example for $C^*$-log-convex sets. Moreover, if $\mathcal{K}\subseteq\mathcal{B}(\mathcal{H})$ is   $C^*$-log-convex, then
 $\mathcal{K}^{-1}=\{X^{-1};\ X\in\mathcal{K}\}$ is convex in the usual sense. For if $X,Y\in\mathcal{K}^{-1}$ and $\lambda\in[0,1]$, with $C_1=\sqrt{\lambda}I$ and $C_2=\sqrt{1-\lambda}I$ we have $C_1^*C_1+C_2^*C_2=I$ and so
 $(C_1^*XC_1+C_2^*YC_2)^{-1}\in\mathcal{K}$. This follows that $\lambda X+(1-\lambda)Y=C_1^*XC_1+C_2^*YC_2\in\mathcal{K}^{-1}$.  More generally, a set  $\mathcal{K}\subseteq\mathcal{B}(\mathcal{H})$ is a $C^*$-log-convex set if and only if $\mathcal{K}\subseteq \mathrm{Inv}(\mathcal{B}(\mathcal{H}))$ and $\mathcal{K}^{-1}$ is a $C^*$-convex set, where we mean by $\mathrm{Inv}(\mathcal{B}(\mathcal{H}))$  the set of invertible elements in $\mathcal{B}(\mathcal{H})$.
\begin{proposition}
  If $\mathcal{L}$ and $\mathcal{K}$ are $C^*$-log-convex sets, then so is
    $$\left(\mathcal{K}^{-1}+\mathcal{L}^{-1}\right)^{-1}=\left\{\left(X^{-1}+Y^{-1}\right)^{-1};\ \ X\in\mathcal{K},\ Y\in\mathcal{L}\right\}.$$
\end{proposition}
\begin{proof}
  Assume that  $C_i\in\mathcal{B}(\mathcal{H})$ with $\sum_{i=1}^{n}C_i^*C_i=I$. If $Z_1,\cdots,Z_n\in\left(\mathcal{K}^{-1}+\mathcal{L}^{-1}\right)^{-1}$, then $Z_i=\left(X_i^{-1}+Y_i^{-1}\right)^{-1}$ for some $X_i\in\mathcal{K}$ and $Y_i\in\mathcal{L}$,\  $(i=1,\cdots,n)$. It follows from the $C^*$-log-convexity of $\mathcal{K}$ and $\mathcal{L}$ that
  $\left(\sum_{i=1}^{n}C_i^*X_i^{-1}C_i\right)^{-1}\in \mathcal{K}$ and $\left(\sum_{i=1}^{n}C_i^*Y_i^{-1}C_i\right)^{-1}\in\mathcal{L}$. Therefore,
  {\small\begin{align*}
 \left(\sum_{i=1}^{n}C_i^*Z_i^{-1}C_i\right)^{-1}&=
  \left(\sum_{i=1}^{n}C_i^*\left(X_i^{-1}+Y_i^{-1}\right)C_i\right)^{-1}\\
  &=
  \left(\sum_{i=1}^{n}C_i^*X_i^{-1}C_i+\sum_{i=1}^{n}C_i^*Y_i^{-1}C_i\right)^{-1}\in\left(\mathcal{K}^{-1}
  +\mathcal{L}^{-1}\right)^{-1},
  \end{align*}}
  which implies that $\left(\mathcal{K}^{-1}+\mathcal{L}^{-1}\right)^{-1}$ is $C^*$-log-convex.
\end{proof}
\begin{proposition}
Let   $\mathcal{K}\subseteq\mathcal{B}(\mathcal{H})$ be inverse closed in the sense  that  $ \mathcal{K}^{-1}\subseteq\mathcal{K}$. If $\mathcal{K}$ is $C^*$-log-convex, then it is $C^*$-convex.
\end{proposition}
\begin{proof}
Assume that   $\mathcal{K}$ is a $C^*$-log-convex set with $ \mathcal{K}^{-1}\subseteq\mathcal{K}$. Let $C_i\in\mathcal{B}(\mathcal{H})$ and  $\sum_{i=1}^{n}C_i^*C_i=I$.  If $X_i\in\mathcal{K}$, then $X_i^{-1}\in\mathcal{K}$ and we have from the $C^*$-log-convexity of $\mathcal{K}$ that
 $\left(\sum_{i=1}^{n}C_i^*X_iC_i\right)^{-1}\in\mathcal{K}$. It follows that $\sum_{i=1}^{n}C_i^*X_iC_i\in\mathcal{K}$ and so $\mathcal{K}$ is $C^*$-convex.
\end{proof}

The convex hull of a set $\mathcal{K}$ in a vector space $\mathcal{X}$ is defined to be the smallest convex set in $\mathcal{X}$ containing $\mathcal{K}$. It is known that  the convex hull of  $\mathcal{X}$ is the set
\begin{align}\label{conv}
{\rm CH}(\mathcal{K})=\left\{\sum_{i=1}^{m}t_ia_i; \ a_i\in \mathcal{K},\ m\in\mathbb{N},  \ \sum_{i=1}^{m}t_i=1\right\}.
\end{align}
The  $C^*$-convex hull \cite{Loebl} of  a  set $\mathcal{K}\subseteq\mathcal{B}(\mathcal{H})$ is the smallest $C^*$-convex set in $\mathcal{B}(\mathcal{H})$ which contains $\mathcal{K}$. This is the generalization of convex hull in the non-commutative setting. It is  known  \cite[Corollary 20]{Loebl} that  given $T\in\mathcal{B}(\mathcal{H})$,  the $C^*$-convex hull of $\{T\}$ is the set
$$C^*{-\rm CH}(T)=\left\{\sum_{i}C_i^*TC_i;\ \sum_{i}C_i^*C_i=I \right\}.$$
Moreover, we define the $C^*$-log-convex hull of a  set $\mathcal{K}\subseteq\mathcal{B}(\mathcal{H})$ to be the smallest $C^*$-log-convex set in $\mathcal{B}(\mathcal{H})$ which contains $\mathcal{K}$. It is easy to see that if $T\in\mathcal{B}(\mathcal{H})$ and $T>0$, then the $C^*$-log-convex hull of $\{T\}$ turns out to be
$$C^*{-\rm LCH}(T)=\left\{\left(\sum_{i}C_i^*T^{-1}C_i\right)^{-1};\ \sum_{i}C_i^*C_i=I \right\}.$$
When  $T\in\mathbb{H}_n$,  we can present the $C^*$-convex hull of $\{T\}$ in terms of  its eigenvalues.
\begin{theorem}\label{th3}
  If $\lambda_1,\cdots,\lambda_n$ are eigenvalues of $T\in\mathbb{H}_n$, then
\begin{align*}
C^*{-\rm CH}(T)=\left\{\sum_{i=1}^{n}\lambda_iE_i;\quad E_i\geq0,\  i=1,\cdots,n ,\ \ \sum_{i=1}^{n}E_i=I\right\}.
  \end{align*}
\end{theorem}
\begin{proof}
Let $T=\sum_{i=1}^{n}\lambda_iP_i$ be the spectral decomposition of $T$.   Put
\begin{align*}
 \Omega=\left\{\sum_{i=1}^{n}\lambda_iE_i;\  E_i\geq0,\  i=1,\ldots,n ,\  \sum_{i=1}^{n}E_i=I\right\}.
    \end{align*}
  If $X\in C^*-{\rm CH}(T)$, then $X=\sum_{i}C_i^*TC_i$ for some $C_i\in\mathbb{M}_n$ with $\sum_{i}C_i^*C_i=I$.   Therefore,
 \begin{align*}
    X=\sum_{i}C_i^*TC_i&= \sum_{i}C_i^*\left(\sum_{j=1}^{n}\lambda_jP_j\right)C_i=\sum_{j=1}^{n}\lambda_j\sum_{i}
    C_i^*P_jC_i.
  \end{align*}
  Putting  $E_j=\sum_{i}C_i^*P_jC_i$, we have $E_j\geq0$,\ $(j=1,\ldots,n)$ and
\begin{align*}
  \sum_{j=1}^{n}E_j=\sum_{j=1}^{n}\sum_{i}C_i^*P_jC_i=\sum_{i}C_i^*\left(\sum_{j=1}^{n}P_j\right)C_i
  =\sum_{i}C_i^*C_i=I.
    \end{align*}
 Hence, $X= \sum_{j=1}^{n}\lambda_jE_j$ and $\sum_{j=1}^{n}E_j=I$, i.e., $X\in\Omega$.

For the converse, note that the $C^*{-\rm CH}(T)$ is $C^*$-convex and contains all eigenvalues of $T$. Now if $X=\sum_{j=1}^{n}\lambda_jE_j$ in which $\sum_{j=1}^{n}E_j=I$ and $E_j\geq0$,\ $(j=1,\cdots,n)$, then
\begin{align*}
X=\sum_{j=1}^{n}\lambda_jE_j=\sum_{j=1}^{n}\sqrt{E_j}\lambda_j\sqrt{E_j}\in C^*{-\rm CH}(T),
\end{align*}
   by $C^*$-convexity of $C^*{-\rm CH}(T)$.
\end{proof}
Let $f:J\to\mathbb{R}$ be a continuous function. If $T\in\mathbb{H}_n$ has the spectral decomposition $T=\sum_{i=1}^{n}\lambda_iP_i$ in which the eigenvalues  $\lambda_1,\cdots,\lambda_n$  are contained in $J$, then the well known functional calculus  yields that $f(T)=\sum_{i=1}^{n}f(\lambda_i)P_i$. By use of Theorem \ref{th3}, the $C^*$-convex hull of $f(T)$ turns to be
\begin{align*}
C^*{-\rm CH}(f(T))=\left\{\sum_{i=1}^{n}f(\lambda_i)E_i;\quad E_i\geq0,\  i=1,\cdots,n ,\ \ \sum_{i=1}^{n}E_i=I\right\}.
  \end{align*}

The next result reveals the reason of naming   $C^*$-log-convex sets. First note that, the notion of $C^*$-log-convexity can be extended to subsets of an algebra with a $\mathcal{B}(\mathcal{H})$-module structure. For example, a   set   $\mathcal{K}\subseteq\mathcal{B}(\mathcal{H})\times\mathcal{B}(\mathcal{H})$  is called  $C^*$-log-convex  if
$(X_i,Y_i)\in\mathcal{K}$, $C_i\in\mathcal{B}(\mathcal{H})$ and  $\sum_{i=1}^{n}C_i^*C_i=I$ implies that
$$\left(\sum_{i=1}^{n}C_i^*(X_i,Y_i)^{-1}C_i\right)^{-1}=\left(\left(\sum_{i=1}^{n}C_i^*X_i^{-1}C_i\right)^{-1},
\left(\sum_{i=1}^{n}C_i^*Y_i^{-1}C_i\right)^{-1}\right)\in\mathcal{K}.$$

\begin{theorem}
  A continuous function $f:(0,\infty)\to(0,\infty)$ is operator log-convex if and only if the set $\mathcal{K}=\{(X,Y)\in\mathcal{B}(\mathcal{H})^+\times\mathcal{B}(\mathcal{H})^+; \ \ f(X^{-1})\leq Y\}$ is a $C^*$-log-convex set.
\end{theorem}
\begin{proof}

 Let $f$ be operator log-convex,   $C_i\in\mathcal{B}(\mathcal{H})$ and
  $\sum_{i=1}^{n}C_i^*C_i=I$. If  $(X_i,Y_i)\in\mathcal{K}$,  then $f\left(X_i^{-1}\right)\leq Y_i$.  It follows from \eqref{jensen2} that
  \begin{align*}
    f\left(\sum_{i=1}^{n}C_i^*X_i^{-1}C_i\right)\leq \left(\sum_{i=1}^{n}C_i^*f\left(X_i^{-1}\right)^{-1}C_i\right)^{-1}\leq \left(\sum_{i=1}^{n}C_i^*Y_i^{-1}C_i\right)^{-1}.
  \end{align*}
  Hence $\left(\left(\sum_{i=1}^{n}C_i^*X_i^{-1}C_i\right)^{-1},\left(\sum_{i=1}^{n}C_i^*Y_i^{-1}C_i\right)^{-1}\right)
\in\mathcal{K}$ and so $\mathcal{K}$ is $C^*$-log-convex.

Conversely, assume that $\mathcal{K}$ is $C^*$-log-convex, $C_i\in\mathcal{B}(\mathcal{H})$ and
  $\sum_{i=1}^{n}C_i^*C_i=I$. If $X_i\in \mathcal{B}(\mathcal{H})^+$\ $(i=1,\cdots,n)$,  then
  $(X_i^{-1},f(X_i))\in \mathcal{K}$. Therefore
  $$\left(\left(\sum_{i=1}^{n}C_i^*X_iC_i\right)^{-1},\left(\sum_{i=1}^{n}C_i^*f(X_i)^{-1}C_i\right)^{-1}\right)
  \in\mathcal{K}$$
  by the $C^*$-log-convexity of $\mathcal{K}$. It follows that
  $$f\left(\sum_{i=1}^{n}C_i^*X_iC_i\right)\leq \left(\sum_{i=1}^{n}C_i^*f(X_i)^{-1}C_i\right)^{-1}$$
  and so $f$ is operator log-convex by \eqref{jensen2}.
\end{proof}


\end{document}